\documentclass[12pt,oneside]{amsart}

\usepackage{amsfonts, amsmath, amssymb, amsthm, amscd, hyperref}

\usepackage{anysize}

\newtheorem{thm}{Theorem}[section]
\newtheorem*{thm*}{Theorem}
\newtheorem{lem}[thm]{Lemma}

\newtheorem{cor}[thm]{Corollary}

\theoremstyle{definition}

\theoremstyle{remark}
\newtheorem{rem}[thm]{Remark}
\newtheorem{prob}[thm]{Problem}

\numberwithin{equation}{section}

\DeclareMathOperator{\inte}{int}

\renewcommand{\epsilon}{\varepsilon}
\renewcommand{\kappa}{\varkappa}

\begin{document}

\title[Geometric coincidence results\dots]{Geometric coincidence results from multiplicity of continuous maps}

\author{R.N.~Karasev}
\thanks{This research is supported by the Dynasty Foundation, the President's of Russian Federation grant MK-113.2010.1, the Russian Foundation for Basic Research grants 10-01-00096 and 10-01-00139, the Federal Program ``Scientific and scientific-pedagogical staff of innovative Russia'' 2009--2013}
\email{r\_n\_karasev@mail.ru}
\address{Roman Karasev, Dept. of Mathematics, Moscow Institute of Physics and Technology, Institutskiy per. 9, Dolgoprudny, Russia 141700}


\begin{abstract}
In this paper we study geometric coincidence problems in the spirit of the following problems by B.~Gr\"unbaum: How many affine diameters of a convex body in $\mathbb R^n$ must have a common point? How many centers (in some sense) of hyperplane sections of a convex body in $\mathbb R^n$ must coincide?

One possible approach to such problems is to find topological reasons for multiple coincidences for a continuous map between manifolds of equal dimension. In other words, we need topological estimates for the multiplicity of a map. In this work examples of such estimates and their geometric consequences are presented.
\end{abstract}

\maketitle

\section{Introduction}

Before going to particular geometric problems, let us state the corresponding topological question, which is itself quite natural:

\begin{prob}
\label{proj-mult-prob}
Suppose $d$ is a positive integer. Find the largest possible $\kappa(d)$ with the following property: For any continuous map $f:\mathbb RP^d\to S^d$, there exists $x\in S^d$ such that $|f^{-1}(x)| \ge \kappa(d)$, where $|f^{-1}(x)|$ is the cardinality of the preimage.
\end{prob}

Since $\mathbb RP^d$ and $S^d$ are not homeomorphic for $d\ge 2$, then obviously $\kappa(d)\ge 2$ for $d\ge 2$. In what follows we mostly rely on the partial answer to this question given in~\cite{kar2010coinc}:

\begin{thm}
\label{proj-mult-thm}
Suppose that $q$ is a power of two, $q(m+1) < 2^l-1$. Then any continuous map
$$
f : \mathbb RP^{2^l-2-m}\to \mathbb R^{2^l-2}
$$
has a coincident $q$-tuple, i.e. $|f^{-1}(x)|\ge q$ for some $x\in\mathbb R^{2^l-2}$.
\end{thm}

\begin{rem}
This theorem also holds for maps $f : \mathbb RP^{2^l-2-m}\to S^{2^l-2}$, because the proof in~\cite{kar2010coinc} uses only the Stiefel--Whitney classes of the domain and the target manifolds. In particular, putting $m=0$ we obtain $\kappa(2^l-2)\ge 2^{l-1}$, in other words $\kappa(d) = \frac{d}{2}+1$ for $d=2^l-2$. For other $m$ we obtain that $\kappa(2^l-2-m)$ is at least the greatest power of two in the range $\left(1, \frac{2^l-1}{m+1}\right)$, which is a nontrivial result if $2^l-1 > 4m + 4$.
\end{rem}

It is an open problem to find more lower bounds for $\kappa(d)$. For example, in the case $d=2^l-1$ Theorem~\ref{proj-mult-thm} cannot give more than $\kappa(d)\ge 1$, because the characteristic classes of $\mathbb RP^{2^l-1}$ vanish (of course we know that $\kappa(d)\ge 2$ for $d\ge 2$). In the paper of M.~Gromov~\cite{grom2010} the cardinalities of preimages were also studied (for piecewise linear or piecewise smooth maps under some genericity conditions), but useful results were obtained for maps from polyhedra of high complexity, very different from $\mathbb RP^d$. It is also noted in~\cite[p.~447]{grom2010} that any $d$-manifold admits a map to $\mathbb R^d$ with multiplicity $|f^{-1}(x)|\le 4d$ for any $x\in\mathbb R^d$, thus giving an upper bound for our function: $\kappa(d)\le 4d$.

The approach to the multiplicity in~\cite{kar2010coinc} uses computations in the cohomology of configuration spaces of manifolds. It may happen that the technique of contraction in the space of (co)cycles from~\cite{grom2010}, after some modifications, can also be applied to Problem~\ref{proj-mult-prob}. In particular, it would be natural to have a linear lower bound $\kappa(d)\ge cd$ for some constant $c>0$, which we do not have at the moment.

In the rest of the paper we produce some geometric consequences of Theorem~\ref{proj-mult-thm} with $S^{2^l-2}$ as the target space.

The author thanks Vladimir Dol'nikov for pointing out the geometrical problems and the discussion.

\section{Continuous selection of lines in every direction}

Consider a family of lines $\Omega$ in $\mathbb R^d$ containing one straight line in every direction, depending continuously on the direction $\ell \in\mathbb RP^{d-1}$. A natural question is the following:

\begin{prob}
Find the largest possible $\kappa_2(d)$ such that we can always find $\ge \kappa_2(d)$ distinct lines in such continuous family $\Omega$ that have a common point.
\end{prob}

\begin{lem}
In the introduced notation $\kappa_2(d)\ge \kappa(d)$
\end{lem}
\begin{proof}
Consider the union $X$ of all lines in $\Omega$ with the natural topology, and take its $1$-point compactification $\tilde X$, which is homeomorphic to $\mathbb RP^d$. The natural map $X\to \mathbb R^d$ is extended to a continuous map 
$$
\tilde f : \tilde X = \mathbb RP^d \to S^d,
$$
and the multiple points of $\tilde f$ give the desired coincident lines. 
\end{proof}

Note that \emph{affine diameters} (longest straight line sections in given direction) of a strictly convex body form a continuous family of lines in every direction, thus proving:

\begin{cor}
Suppose $K\subset \mathbb R^d$ is a strictly convex body. Then there exist at least $\kappa_2(d)\ge \kappa(d)$ distinct affine diameters of $K$, having a common point.
\end{cor}

This result can be sharpened; the multiple point may be found in the interior:

\begin{thm}
Suppose $K\subset \mathbb R^d$ is a strictly convex body. Then there exist at least $\kappa(d)$ distinct affine diameters of $K$, having a common point in the interior of $K$.
\end{thm}

\begin{rem}
This problem (along with some other problems) on affine diameters were mentioned in~\cite[\S~6.5]{gr1963}. It was conjectured that some $d+1$ affine diameters have a common point, but our Theorem~\ref{proj-mult-thm} gives $\frac{d}{2}+1$ in the best case $d=2^l-2$.
\end{rem}

\begin{proof}
For every direction $\ell\in\mathbb RP^{d-1}$ consider the affine diameter $a(\ell)$ of $K$ and the open segment $b(\ell) = a(\ell)\cap\inte K$. The union of all $b(\ell)$ is homeomorphic to the total space of the canonical bundle $\gamma: E(\gamma) \to \mathbb RP^{d-1}$. We obtain an obvious map 
$$
f : E(\gamma) \to \inte K.
$$
Note that the one-point compactification of $E(\gamma)$ is homeomorhic to $\mathbb RP^d$, and the one-point compactification of $\inte K$ is homemorphic to $S^d$. Since $f$ is a proper map, we extend it to
$$
\tilde f : \mathbb RP^d\to S^d
$$
and apply the definition of $\kappa(d)$.
\end{proof}

\section{Continuous selection of points in hyperplanes}

There is another geometric problem, that has something to do with Problem~\ref{proj-mult-prob}. Denote $M(d,k)$ the set of all affine $k$-planes in $\mathbb R^d$ with natural topology.

\begin{prob}
Consider continuous selections of a point in a hyperplane, i.e. continuous functions $f : M(d, d-1)\to \mathbb R^d$ such that 
$$
\forall H\in M(d,d-1)\ f(H)\in H.
$$ 
Find the largest possible $\kappa_3(d)$ such that for any continuous selection $f$ we can find at least $\kappa=\kappa_3(d)$ distinct hyperplanes $H_1,\ldots, H_\kappa$ such that 
$$
f(H_1)=\dots =f(H_\kappa).
$$
\end{prob}

\begin{lem}
In the introduced notation $\kappa_3(d)\ge \kappa(d)$.
\end{lem}

\begin{proof}
By the standard duality reasoning, if we denote $H_\infty$ the hyperplane at infinity then $M(d, d-1)\cup\{H_\infty\} = \mathbb RP^d$. The map $f$ is extended by continuity to the map $\tilde f: \mathbb RP^d\to S^d$, and the result follows.
\end{proof}

Let us state a particular geometric consequence of the above lemma:

\begin{thm}
Let $c(L)$ be a motion-invariant continuous selection of a point in an $(d-1)$-dimensional convex compact $L$, e.g. $c(L)$ is the mass center, or the Steiner point of $L$ etc. Suppose $K\subset \mathbb R^d$ is a strictly convex body. Then we can find at least $\kappa=\kappa_3(d)\ge \kappa(d)$ distinct hyperplane sections $L_1,\ldots, L_\kappa$ of $K$ such that
$$
c(L_1) = \dots = c(L_\kappa).
$$
\end{thm}

\begin{rem}
In the particular case of the mass center the stronger result is known: $d+1$ sections with a common mass center can be found, see~\cite[\S~6.1--6.2]{gr1963}. Again, our Theorem~\ref{proj-mult-thm} gives $\frac{d}{2}+1$ in the best case $d=2^l-2$.
\end{rem}

\begin{proof}
For any hyperplane $H\in M(d,d-1)$ define $f(H)$ as follows:
\begin{itemize}
\item 
If $H\cap K\neq \emptyset$ then $f(H) = c(H\cap K)$.
\item 
If $H\cap K=\emptyset$ then $f(H)$ is the closest to $K$ point in $H$.
\end{itemize}

The map $f$ is a continuous selection of a point in a hyperplane, therefore some $\kappa=\kappa_3(d)$ of hyperplanes have a common point $x$. Consider several cases:

\begin{itemize}
\item
If $x\not\in K$ then $f(H)=x$ holds for the only one hyperplane, orthogonal to the projections vector from $x$ to $K$, so this is not the case;
\item
If $x\in \partial K$ then $f(H)=x$ is valid for the tangent hyperplane at $x$ to $K$ only;
\item
If $x\in\inte K$ then we have the desired statement.
\end{itemize}
\end{proof}

\end{document}